\documentclass[11pt]{article}
\usepackage{epsf}
\usepackage{amsbsy,amsmath}
\usepackage{mathtools}
\usepackage{mathrsfs}
\usepackage{amsfonts}
\usepackage{breqn}
\usepackage{amssymb}
\usepackage{enumitem}
\usepackage{eucal}
\usepackage{graphics,mathrsfs}
\usepackage{amsthm}
\usepackage{amssymb}
\usepackage{secdot}
\newtheorem{theorem}{Theorem}[section]
\newtheorem{lemma}[theorem]{Lemma}
\newtheorem{proposition}[theorem]{Proposition}

\theoremstyle{definition}
\newtheorem{definition}[theorem]{Definition}

\theoremstyle{remark}
\newtheorem{remark}[theorem]{Remark}
\numberwithin{equation}{section}

\numberwithin{equation}{section}

\newsavebox{\savepar}

\begin{document}

\title{Existence of solution for a system involving fractional Laplacians and a Radon measure}
\author{Amita Soni and D.Choudhuri}

\date{}
\maketitle

\begin{abstract}
\noindent An existence of a nontrivial solution in some `weaker' sense of the following system of equations
\begin{align*}
(-\Delta)^{s}u+l(x)\phi u+w(x)|u|^{k-1}u&=\mu~\text{in}~\Omega\nonumber\\
(-\Delta)^{s}\phi&= l(x)u^2~\text{in}~\Omega\nonumber\\
u=\phi&=0 ~\text{in}~\mathbb{R}^N\setminus\Omega
\end{align*}
 has been proved. Here $s \in (0,1)$, $l,w$ are bounded nonnegative functions in $\Omega$, $\mu$ is
a Radon measure and $k > 1$ belongs to a certain range.
\end{abstract}
\begin{flushleft}

\end{flushleft}

\section{Introduction}
Fractional Calculus is a new tool which has been off-late employed to model difficult biological systems with nonlinear behavior. The notion of a fractional calculus came into being to answer some simple questions which were related to the notion of derivatives such as, the first order derivative represents the slope
of a function, what does a half-an-order derivative of a function geometricallly mean?. In a quest to seek answers to such questions, a new avenue of a bridge between the mathematical and the real world was discovered, which led to many questions besides its answers.\\
Meanwhile, with the rapid advancement in the field of elliptic PDE, one of the leading subject of interest for researchers in Mathematics are elliptic problems involving measure data. The presence of a measure data in the problem makes it difficult to apply any well known variational methods to prove the existence of solution(s). Some remarkable works to deal with such kind of situations can be seen in \cite{veron}, \cite{bocc}, \cite{Ponce}, \cite{Marcus}, \cite{Vecchio}, \cite{Yang} and  the references therein. In \cite{bocc}, the authors have proved the existence of a weak solution of the problem involving a positive Radon measure and a Carath\'{e}odory function which are assumed to satisfy certain conditions. In \cite{Vecchio}, the author has showed the existence of weak solutions of problem involving two caratheodory functions with right hand side a bounded Radon measure. In \cite{Ponce}, the authors have \textit{determined} the reduced limit to the nonhomogeneous part of a semilinear problem with the Laplacian operator which is a Radon measure. The readers may further refer to the book due to Marcus and V\'{e}ron \cite{Marcus} which may be used as a ready reckoner to concepts on Elliptic PDEs with measure datum. In \cite{Yang}, the authors have studied the existence of nontrivial weak solutions in a general regular domain which is not necessarily bounded for a fractional Laplacian operator. Recently,  Chen and V\'{e}ron \cite{veron},  have proved the existence and uniqueness of a very weak solution of a fractional Laplacian problem involving a Radon measure and also showed that absolutely continuity of this measure with respect to some Bessel capacity is a necessary and sufficient condition for the existence of a very weak nontrivial solution. Since the current work is on a system of equations, which resembles a Schr\"{o}dinger-Poisson system, hence it is customary to refer to some important works on a Schr\"{o}dinger-Poisson system of equations can be found in \cite{zhang}, \cite{cingo}, \cite{dimi} and the references therein. Zhang et al \cite{zhang} have studied the nonlinear Schr\"{o}dinger-Poisson system and have proved the existence of positive solution over $\mathbb{R}^3$.  Dimitri Mugnai \cite{dimi} has studied the solitary waves of a nonlinear Schr\"{o}dinger-Poisson system and have guaranteed the existence of radially symmetric solution over $\mathbb{R}^3$. Further, Cingolani et al \cite{cingo} has guaranteed in the existence of high energy solution over $\mathbb{R}^2$. Motivated by \cite{veron}, in this paper we considered a system of PDEs which is as follows. 
\begin{align}\label{maineqn}
A:~(-\Delta)^{s}u+l(x)\phi u+w(x)|u|^{k-1}u&=\mu~\text{in}~\Omega\nonumber\\
B:~~~~~~~~~~~~~~~~~~~~~~~~~~~~~~~~~(-\Delta)^{s}\phi&= l(x)u^2~\text{in}~\Omega\nonumber\\
u=\phi&=0 ~\text{in}~\mathbb{R}^N\setminus\Omega
\end{align}
has been proved. The first equation in the system defined in (\ref{maineqn}) will be denoted as `problem A' and the second equation as `problem B.' Here $s \in (0,1)$, $l,w$ are bounded nonnegative functions in $\Omega$, $\mu$ is
a Radon measure and $k > 1$ belongs to a certain range.
We will prove the existence and uniqueness of a nontrivial, solution to the  system of equations (\ref{maineqn}) in a weaker sense which will be defined in the succeeding section. Further, we will also prove a necessary and sufficient condition for the existence of a solution. 
\section{Important results and definitions}
We state a few definitions, lemmas, theorems and propositions along
with the notations which will be consistently used by us in the
succeeding section(s).
\begin{definition}
For $p\in[1,\infty)$, the fractional Sobolev space $W^{s,p}(\Omega)$ is defined as $$W^{s,p}(\Omega)=\left\lbrace u\in L^{p}(\Omega): \frac{|u(x)-u(y)|}{|x-y|^{\frac{n}{p}+s}}\in L^{p}(\Omega\times\Omega)\right\rbrace$$ with the norm 
$$||u||_{s,p}=\left(\int_{\Omega}|u|^{p}dx+\int_{\Omega}\int_{\Omega}\frac{|u(x)-u(y)|^{p}}{|x-y|^{n+sp}}dx dy\right)^{\frac{1}{p}}.$$
\end{definition}
\noindent We now give the definition of a `{\it very weak solution}' as defined in \cite{veron}.
\begin{definition}\label{vws}
We say that $u$ is a very weak solution of the problem
\begin{align}
(P):~~(-\Delta)^{s}u+g(u)&=\mu~\text{in}~\Omega\nonumber\\
u&=0 ~\text{in}~\mathbb{R}^N\setminus\Omega, \label{ineq2}
\end{align}
if $u\in L^1(\Omega)$, $g(u)\in L^1(\Omega,\rho^s dx)$ and
\begin{equation}
\int_{\Omega}[u(-\Delta)^s\xi+g(u)\xi]dx=\int_{\Omega}\xi
d\mu\label{ineq2'}
\end{equation}
$\forall\xi\in X_s$ and $X_{s}\subset C(\mathbb{R}^N)$ satisfying
the following.
\begin{enumerate}
\item $supp(\xi)\subset\overline{\Omega}$
\item $(-\Delta)^s\xi$ exists for all $x\in\Omega$ and
$|(-\Delta)^s\xi|\leq C$ for some $C>0$,
\item  $\exists\;\phi\in
L^{1}(\Omega,\rho^{s}dx)$, $\epsilon_0>0$ such that
$|(-\Delta)_{\epsilon}^s\xi|\leq \phi$ a.e. in $\Omega$,
$\forall\epsilon\in (0,\epsilon_0]$. Here
$(-\Delta)^s_{\epsilon}u(x)=-C(N,s)\int\frac{u(z)-u(x)}{|z-x|^{N+2s}}\chi_{\epsilon}(|x-z|)dz$.
\end{enumerate}
Here $\mu\in\mathfrak{m}(\Omega,\rho^{\beta})$ is a Radon measure
for $0 \leq \beta \leq s$, $0<s<1$.
\end{definition}
\begin{definition}
The critical exponent is defined as
\[k(s,\beta)=\begin{cases}
      \frac{N}{N-2s} & \beta\in [0,\frac{N-2s}{N}s), \\
      \frac{N+s}{N-2s+\beta} & \beta\in
      (\frac{N-2s}{N}s,s]
    \end{cases}\]
    for $N\geq 2$, $0<s<1$, $0<\beta<s$ and $g(.)\in L^1(\Omega,\rho^s)$ is a nonlinear function such hat $g(0)=0$.\\
\end{definition}
\begin{definition}
Let $\Omega\subset\mathbb{R}^N$ be a domain and $\mu$ be a positive
Borel measure in $\Omega$. For $k>1$, $k'=\frac{k}{k-1}$ and $u\in
L_{loc}^1(\Omega,d\mu)$, we define the {\it Marcinkiewicz space} as
\begin{equation}
M^k(\Omega,d\mu)=\{u\in
L_{loc}^{1}(\Omega,d\mu):||u||_{M^k(\Omega,d\mu)}<\infty\}\nonumber
\end{equation}
where $||u||_{M^k(\Omega,d\mu)}=\inf\{c\in
[0,\infty]:\int_{A}|u|d\mu\leq
c\left(\int_{A}d\mu\right)^{1/k'},\forall A\subset\Omega~\text{Borel
set}\}$.
\end{definition}
\noindent The following propositions from \cite{veron} will play a crucial role in this work.
\begin{proposition}
If $f\in C^{\gamma}(\overline{\Omega})$ for $\gamma > 0$, $\exists$
a very weak solution $u\in X_s$ of the problem
\begin{align}
(-\Delta)^s u&= f~\text{in}~\Omega\nonumber\\
u&=0~\text{in}~\mathbb{R}^N\setminus\Omega.
\end{align}
\begin{proposition}
If $f\in L^1(\Omega,\rho^s dx)$, there exists a unique weak solution
$u$ of the problem
\begin{align}
(-\Delta)^s u&=f~\text{in}~\Omega\nonumber\\
u&=0~\text{in}~\mathbb{R}^N\setminus\Omega.
\end{align}
For any $\xi\in X_s$, $\xi\geq 0$, we have
\begin{equation}
\int_{\Omega}|u|(-\Delta)^s\xi dx\leq
\int_{\Omega}\xi\text{sign}(u)fdx
\end{equation}
\begin{equation}
\int_{\Omega}u_{+}(-\Delta)^s\xi dx\leq
\int_{\Omega}\xi\chi_{\{x:u(x)\geq 0\}}fdx.
\end{equation}
\end{proposition}
\end{proposition}
\begin{remark}
The central idea is to reduce the system of equations to a scalar equation consisting of one unknown and guarantee the existence of a solution in the sense of Definition \ref{vws}, i.e., in a very weak sense.
\end{remark}
\begin{definition}
We will define a nontrivial solution to the system of PDEs in (\ref{maineqn}) as a pair $(u,\phi)$ if $u\neq 0$, $\phi\neq 0$ and $(u,\phi)$ {\it solves} (\ref{maineqn}).
\end{definition}
\section{Main results}
\noindent We consider the system of PDEs
\begin{align}
(-\Delta)^{s}u+l(x)\phi u+w(x)|u|^{k-1}u&=\mu~\text{in}~\Omega\nonumber\\
(-\Delta)^{s}\phi&= l(x)u^2~\text{in}~\Omega\nonumber\\
u=\phi&=0 ~\text{in}~\mathbb{R}^N\setminus\Omega,\label{ineq1}
\end{align}
where $s \in (0,1)$, $l,w$ are bounded nonnegative functions in $\Omega\subset\mathbb{R}^N$, $\mu$ is
a Radon measure, $1 \leq k \leq k_{s,\beta}$. We use the notations
used by Chen and V\'{e}ron \cite{veron} in their paper. 
The sense of
The system can be converted to a scalar equation of the type in \cite{veron} if one uses the following representation, due to \cite{stinga}, for $\phi$ in terms of $u$.
\begin{equation}
\phi_{u}(x)=C(N,s)\int_{\Omega}\frac{l(y)(u(y))^2}{|x-y|^{N-2s}}dy.\label{ineq4}
\end{equation}
Thus ($\ref{ineq1}$) can be expressed as
\begin{align}
(-\Delta)^{s}u+F[u]+w(x)|u|^{k-1}u&=\mu~\text{in}~\Omega\nonumber\\
u&=0 ~\text{in}~\mathbb{R}^N\setminus\Omega \label{ineq5}
\end{align}
where $F[u]=l(x)\phi_u(x)u$. As stated in \cite{veron}, the problem
\begin{align}
(-\Delta)^{s}u+g(u)&=\mu~\text{in}~\Omega\nonumber\\
u&=0 ~\text{in}~\mathbb{R}^N\setminus\Omega \label{ineq2}
\end{align}
where $g$ is a continuous, non decreasing function satisfying
$rg(r)\geq 0$ $\forall r\in\mathbb{R}$ and
$\int_{1}^{\infty}(g(t)-g(-t))t^{-1-k_{s,\beta}}dt<\infty$ admits a
unique very weak solution $u_{\mu}$ corresponding to
$\mu\in\mathfrak{m}(\Omega,\rho^{\beta})$. Further
\begin{equation}
-\mathbb{G}[\mu_{-}]\leq u_{\mu} \leq \mathbb{G}[\mu_{+}]~\text{a.e. in}~\Omega \label{ineq3'}
\end{equation}
where $\mu_{-}$, $\mu_{+}$ are the positive and the negative parts of the Jordan decomposition of $\mu$. Note that when $g(x,u)=u+F[u]+w(x)|u|^{k-1}u$, in (\ref{maineqn}), it satisfies the assumptions made on $g$ in \cite{veron}. We now state the Theorem 1.1 in \cite{veron}.
\begin{theorem}
[Theorem 1.1, \cite{veron}] Assume that $\Omega\subset\mathbb{R}^N$	($N \geq 2$) is an open bounded $C^2$ domain, $\alpha\in (0,1)$, $\beta\in[0,\alpha]$ and $k_{\alpha,\beta}$ is defined by . Let $g:\mathbb{R}\rightarrow\mathbb{R}$ be a continuous nondecreasing function satisfying $g(r)r\geq 0$, $\forall r\in\mathbb{R}$ and $\int_{1}^{\infty}(g(s)-g(-s))s^{-1-1k_{\alpha,\beta}}ds<\infty$. Then for any $\nu\in\mathfrak{m}(\Omega,\rho^{\beta})$, problem (\ref{ineq2}) admits a unique weak solution $u_{\nu}$. Furthermore, the mapping: $\nu\mapsto u_{\nu}$ is increasing and $\mathbb{G}[\nu_{-}]\leq u_{\nu}\leq \mathbb{G}[\nu_{+}]$ a.e. in $\Omega$, where $\nu_{+}$ and $\nu_{-}$ are respectively the positive and negative part in the Jordan decomposition of $\nu$.
\end{theorem}
\noindent The $\mathbb{G}[.]$ is the notation for the Green's operator.
\begin{remark}
Note that, whenever we say a solution exists it will always mean in the {\it very weak} sense as in Definition 
\end{remark}
The main results proved in this paper are as follows.
\begin{theorem}
The problem (3.1) admits a unique very weak solution $(u,\phi)$ corresponding to
$\mu\in\mathfrak{m}(\Omega,\rho^{\beta})$. Further
\begin{equation}
-\mathbb{G}[\mu_{-}]\leq u \leq \mathbb{G}[\mu_{+}]~\text{a.e.
in}~\Omega \label{ineq3'}
\end{equation}
where $\mu_{-}$, $\mu_{+}$ are the positive and the negative parts of the Jordan decomposition of $\mu$.
\end{theorem}

\begin{theorem}
$s$, $\Omega$, $k$ are as in problem ($\ref{ineq1}$). Then the
problem ($\ref{ineq2}$) has a solution with a nonnegative bounded
measure $\mu$ iff $\mu$ satisfies on compact subsets of $\Omega$
$\text{Cap}_{s,k'}(K)=0\Rightarrow\mu(K)=0$. Here
$\text{Cap}_{s,k'}(K)=\inf\{||\phi||_{W^{2s,k'}(\Omega)}^{k'}:\phi\in
C_{c}^{\infty}(\Omega), 0\leq \phi\leq 1, \phi\equiv
1~\text{on}~K\}$.
\end{theorem}


\section{Existence and uniqueness}
In order to prove the Theorem 3.1, we first state and prove the following lemma.
\begin{lemma}
Let $\Omega$ be a bounded domain in $\mathbb{R}^N$ with sufficiently
smooth boundary and $g:\mathbb{R}\rightarrow\mathbb{R}$ is
continuous and non decreasing with $rg(r)\geq 0$ $\forall
r\in\mathbb{R}$. Then for any $f\in L^1(\Omega,\rho^{s}dx)$, there
exists a unique very weak solution to $(\ref{ineq1})$.
\end{lemma}
\begin{proof}
We use a variational technique to guarantee an existence to a
solution to the problem in $(\ref{ineq1})$. To attempt this we define $I:W_c^{2,s}(\Omega)\rightarrow \mathbb{R}$ the functional as follows.
\begin{equation}
I(u)=\frac{1}{2}\int_{\Omega}((-\Delta)^su)^2dx+\int_{\Omega}H(x,u)dx
+\Phi[u],\label{ineq6}
\end{equation}
where $\Phi:W_c^{2,s}(\Omega)\rightarrow\mathbb{R}$,
$W_c^{2,s}(\Omega)=\{\underline{w}\in
L^2(\Omega):\int_{\mathbb{R}^N}|\hat{\underline{w}}|^2(1+|x|^{s})dx<\infty\}$
with $\underline{w}$ is the extension of $w\in L^2(\Omega)$ by $0$,
$\Phi[u]=c(N,s)\int_{\Omega}\int_{\Omega}\frac{l(x)l(y)(u(x))^2(u(y))^2}{|x-y|^{N+2s}}dx$,
$H(x,u)$ is the primitive of $h(x,u)=w(x)|u|^{k-1}u$.\\
The functional $I$ is coercive over $W_c^{2,s}(\Omega)$ because
$\frac{1}{2}\int_{\Omega}((-\Delta)^su)^2dx+\int_{\Omega}H(x,u)dx$
is coercive and the coercivity of $\Phi[u]$ can be guaranteed by the fibre maps. Further, the
subdifferential $\partial I$ of the map $I$ is maximal-monotone in
the sense of Browder-Minty (refer \cite{minty} and the references
therein). This can be guaranteed by noting that $\Phi[u]$ is continuous and hence
hemicontinuous. Further, $\Phi[u]$ coercive. Hence by \textit{Browder-Minty} \cite{MB} the range of
$\partial I$ is $L^2(\Omega)$. So for $f\in L^2(\Omega)$ there
exists $u\in W_c^{s,2}(\Omega)$ in the domain of $\partial I$, the
subdifferential of $I$, such that $\partial I(u)=f$, i.e.,
$(-\Delta)^su+l(x)\phi_u(x)u+w(x)|u|^{k-1}u=f$.\\
If $f\in L^1(\Omega,\rho^{s}dx)$, we define
$f_n=\text{sign}(f)\min\{|f|,n\}\in L^2(\Omega)$. We denote the
corresponding solution as $u_n$. Thus we have
\begin{align}
(-\Delta)^su_n+g(x,u_n)&=f_n~\text{in}~\Omega\nonumber\\
u_n&=0~\text{in}~\mathbb{R}^N\setminus\Omega.\nonumber
\end{align}
By virtue of the fact that the PDE $(-\Delta)^s u=1$ with a
homogeneous, Dirichlet boundary condition has a solution, say $u_0$,
we have an estimate (please refer to the appendix in \cite{veron})
\begin{align}\label{estimate1}
||u_1-u_2||_{L^{1}(\Omega)}+||g(x,u_1)-g(x,u_2)||_{L^1(\Omega,\rho^{s}dx)}&\leq
||f_1-f_2||_{L^1(\Omega)},
\end{align}
we see that $(u_n)$, $(g(u_n))$ are
Cauchy sequences in $L^1(\Omega)$, $L^1(\Omega,\rho^sdx)$
respectively. So $u_n \rightarrow u$, $g(u_n) \rightarrow v$ in
$L^1(\Omega)$, $L^1(\Omega,\rho^sdx)$. Therefore there exists a
subsequence, which we still denote as $u_n$, converges to $u$ a.e.
in $\Omega$ and hence $g(u_n)\rightarrow g(u)$ a.e. in $\Omega$. So
$u$ is a very weak solution to the PDE in ($\ref{ineq1}$) with $f\in
L^1(\Omega,\rho^sdx)$. Uniqueness follows from the estimate in (\ref{estimate1}). We now
state an auxiliary lemma as in \cite{veron}.
\begin{lemma}
If $g:[0,\infty)\rightarrow [0,\infty)$ with the assumptions on $g$
as before and $k_{s,\beta}>1$ then
$\underset{t\to\infty}{\lim} g(t)t^{-k_{s,\beta}}=0$.
\end{lemma}
\noindent Continuing with the proof of the Theorem 1.1, suppose
$\mu$ is a Radon measure. Let $C_{\beta}(\overline{\Omega})=\{\zeta\in
C(\overline{\Omega}):\rho^{-\beta}\zeta\in C(\overline{\Omega})\}$
with the norm as
$||\zeta||_{C_{\beta}(\overline{\Omega})}=||\rho^{-\beta}\zeta||_{C(\overline{\Omega})}$.
Consider a sequence of measure $(\mu_n)\subset
C^{1}(\overline{\Omega})$ such that
$\int_{\overline{\Omega}}\zeta\mu_n
dx\rightarrow\int_{\overline{\Omega}}\zeta d\mu$ as
$n\rightarrow\infty$ for each $\zeta\in
C_{\beta}(\overline{\Omega})$. One can conclude from the notion of
convergence of the measures that
$||\mu_n||_{L^1(\Omega,\rho^{\beta}dx)}\leq
c^{\ast}||\mu||_{\mathfrak{m}(\Omega,\rho^{\beta})}$. Then,
\begin{equation}
\int_{\overline{\Omega}}(|u_{n}|+|g(u_{n})|\eta_{1})dx\leq
c'(\int_{\Omega}|\mu_{n}|\rho^{s}dx)\leq
c''||\mu_{n}||_{L^{1}(\Omega,\rho^{\beta}dx)}\leq
c'''||\mu||_{\mathfrak{m}(\Omega,\rho^{\beta})}.\label{ineq7}
\end{equation}
Note that here all constants are positive. $\eta_1$ is a solution to $(-\Delta)^su=1$ with homogeneous, Dirichlet bondary condition that satisfy $c^{\prime{-1}}\leq
\frac{\eta_1}{\rho^{s}}\leq c'$ in $\Omega$ by \cite{serra}. Thus,
$||g(u_n)||_{L^1(\Omega,\rho^s dx)}\leq
c_0||\mu||_{\mathfrak{m}(\Omega,\rho^{\beta})}$ .\\
For $\epsilon > 0$, define
$\xi_{\epsilon}=(\eta_1+\epsilon)^{\frac{\beta}{s}}-\epsilon^{\frac{\beta}{s}}$.
Then by the following lemma given in \cite{veron}
\begin{lemma}
Assume that $u\in X^{s}$ and $\gamma$ is $C^2$ in the interval $u(\overline{\Omega})$ satisfying $\gamma(0)=0$. Then $\gamma\circ u\in X^s$ and for all $x\in\Omega$, there exists $z_{x} \in \overline{\Omega}$ such that
$$(-\Delta)^{s}(\gamma\circ u)(x)=(\gamma^{\prime}\circ u)(x)(-\Delta)^{s}u(x)-\frac{\gamma^{"}\circ u(z_{x})}{2}\int_{\Omega}\frac{(u(x)-u(y))^2}{|y-x|^{N+2s}}dy$$.
\end{lemma}
\noindent we have,
\begin{equation}
\int_{\Omega}(|u_n|\rho^{\beta-s}+|g(u_n)|\rho^{\beta})dx\leq
d_0||\mu_n||_{L^1(\Omega,\rho^{\beta}dx)}\leq
d_{1}||\mu||_{\mathfrak{m}(\Omega,\rho^{\beta})}.\label{ineq8}
\end{equation}
where $d_{0},d_{1} > 0$. This shows that $||g(u_n)||_{L^{1}(\Omega,\rho^{\beta}dx)}\leq
d_2||\mu||_{\mathfrak{m}(\Omega,\rho^{\beta})}$ where $d_{2}$ is positive and independent of $n$. Since,
$u_n=\mathbb{G}[\mu_n-g(u_n)]$ and $(\mu_n-g(u_n))$ is uniformly
bounded in $L^1(\Omega,\rho^{\beta}dx)$, we have
\begin{equation}
||u_n||_{M^{k_{s,\beta}}(\Omega,\rho^s dx)}\leq
||\mu_n-g(u_n)||_{L^1(\Omega,\rho^{\beta}dx)}\leq
c_1||\mu||_{\mathfrak{m}(\Omega,\rho^{\beta})},\label{ineq9}
\end{equation}
where $c_{1} > 0$ and $M^{k_{s,\beta}}$ is the Marcinkiewicz space. Now, by
proposition 2.6 in \cite{veron} which says that the map $f \mapsto
\mathbb{G}[f]$ is compact from $L^1(\Omega,\rho^{\beta}dx)$ into
$L^q(\Omega)$ for any $q\in [1,\frac{N}{N+\beta-2s})$, we have
$(u_n)$ has a strongly convergent subsequence in $L^q(\Omega)$.
Therefore, there exists a subsequence, which we still denote as
$(u_n)$ in $L^1(\Omega)\bigcap L^q(\Omega)$ which converges to $u$ in $L^{q}(\Omega)$. Thus
$g(u_n)\rightarrow g(u)$ a.e. in $\Omega$.\\
We will now see that the sequence $(g(u_n))$ is uniformly
integrable. Define $\tilde{g}(r)=g(|r|)-g(-|r|)$. We see that
$|g(r)|\leq \tilde{g}(|r|)$ for all $r\in\mathbb{R}$. We note that
the operator $\Phi$ being even contributes nothing to $\tilde{g}$.
For each $\lambda>0$ define
$E_{\lambda}=\{x\in\Omega:|u_{n}(x)|>\lambda\}$ and
$\omega(\lambda)=\int_{E_{\lambda}}\rho^{s}dx$. Then for any Borel
set $A$ of $\Omega$ consider,
\begin{align}
\int_{A}|g(u_n)|\rho^{s}dx& =  \int_{A\bigcap
E_{\lambda}^c}|g(u_n)|\rho^{s}dx+\int_{E_{\lambda}}|g(u_n)|\rho^{s}dx\nonumber\\
&\leq 
\tilde{g}(\lambda)\int_{A}\rho^{s}dx+\omega(\lambda)\tilde{g}(\lambda)+\int_{\lambda}^{\infty}\omega(t)d\tilde{g}(t).\label{ineq10}
\end{align}
Further, from the following Proposition in \cite{veron}
\begin{proposition}
Assume that $1\leq q < k < \infty$ and $u\in L^{1}_{loc}(\Omega,d\mu)$. Then there exists $C(q,k) > 0$ such that $$\int_{E}|u|^{q}d\mu \leq C(q,k)||u||_{M^{k}(\Omega,d\mu)}\left(\int_{E}d\mu\right)^{1-\frac{q}{k}}$$ for any Borel set E of $\Omega$.
\end{proposition}
\noindent we have,
\begin{align}
\omega(\lambda)\tilde{g}(\lambda)+\int_{\lambda}^{T}\omega(t)d\tilde{g}(t)&\leq
d_2\tilde{g}(\lambda)\lambda^{-k_{s,\beta}}+d_2\int_{\lambda}^{T}t^{-k_{s,\beta}}d\tilde{g}(t)\nonumber\\
&\leq
d_2T^{-k_{s,\beta}}\tilde{g}(T)+\frac{d_2}{k_{s,\beta}+1}\int_{\lambda}^{T}t^{-1-k_{s,\beta}}\tilde{g}(t)ds.\nonumber
\end{align}
By the Lemma 1.2, we have
\begin{equation}
\omega(\lambda)\tilde{g}(\lambda)+\int_{\lambda}^{\infty}\omega(t)d\tilde{g}(t)\leq\frac{d_2}{k_{s,\beta}+1}
\int_{\lambda}^{\infty}t^{-1-k_{s,\beta}}\tilde{g}(t)ds.\nonumber
\end{equation}
The second term on the right hand side goes to $0$ as
$\lambda\rightarrow\infty$. Hence, for any $\epsilon>0$
$\exists\;\lambda_0$ such that $\frac{d_2}{k_{s,\beta}+1}
\int_{\lambda}^{\infty}t^{-1-k_{s,\beta}}\tilde{g}(t)ds<\epsilon$
$\forall\lambda\geq\lambda_0$. Hence for a fixed
$\lambda\geq\lambda_0$ and from equation ($\ref{ineq10}$) we obtain
$\delta>0$ such that $\int_{A}\rho^sdx\leq\delta$ implies
$\tilde{g}(\lambda)\int_{A}\rho^sdx<\epsilon$.\\
Thus we have for any $\epsilon>0$ $\exists\;\delta >0$ such that
$\int_{A}|g(u_n)|\rho^{s}dx<2\epsilon$ for any Borel set $A$ whose
measure is less than $\delta$ implying that $(g(u_n))$ is uniformly
integrable. In addition we also have that $g(u_n)\rightarrow g(u)$
a.e. in $\Omega$. Hence, by the Vitali convergence theorem we have
$g(u_n)\rightarrow g(u)$ in $L^1(\Omega,\rho^{s}dx)$. In fact, the
result holds for any $0\leq\beta\leq s$. Thus passing the limit
$n\rightarrow\infty$ to
\begin{equation}
\int_{\Omega}(u_n(-\Delta)^s\xi+\xi
g(u_n))dx=\int_{\Omega}\xi\mu_ndx\nonumber
\end{equation}
we obtain
\begin{equation}
\int_{\Omega}(u(-\Delta)^s\xi+\xi g(u))dx=\int_{\Omega}\xi
d\mu.\nonumber
\end{equation}
Thus $u$ is a weak solution to the scalar equation (3.6) and
uniqueness follows for the estimate. So what we have proved is that
the Schr\"{o}dinger-Poisson system of equation in (3.1) has a unique
solution corresponding to a Radon measure in
$\mathfrak{m}(\Omega,\rho^{\beta})$.
\end{proof}
\section{A necessary and sufficient condition}
We now prove Theorem 2.2 which is the necessary and sufficient
condition for the existence of a solution to the problem.
\begin{proof}
{\it Necessary condition}~ Suppose $u$ is a very weak solution of
($\ref{ineq1}$) and let $K$ be a compact subset of $\Omega$. Let
$\phi\in C_c^{\infty}(\mathbb{R}^N)$ such that $0\leq\phi\leq 1$ and
$\phi(x)=1$ over $K$. Set $\xi=\phi^{k'}\in X_{s}$ and
\begin{equation}
\int_{\Omega}(u(-\Delta)^s\xi+g(u)\xi)dx=\int_{\Omega}\xi
d\mu.\label{ineq11}
\end{equation}
Clearly, $\xi\geq\chi_{K}$. It follows from Lemma 4.3,
\begin{equation}
\int_{\Omega}(k'\phi^{k'-1}u(-\Delta)^{s}\phi+\phi^{k'}g(u))dx\geq\mu(K).\label{ineq12}
\end{equation}
By the H\"{o}lder's inequality we have
\begin{equation}
|\int_{K}\phi^{k'-1}u(-\Delta)^{s}\phi dx|\leq
\left(\int_{\Omega}\phi^{k'}u^k
dx\right)^{1/k}\left(\int_{\Omega}|(-\Delta)^s\phi|^{k'}\right)^{1/k'}.\label{ineq13}
\end{equation}
By the equivalence of the norms of $W^{2s,k'}(\Omega)$ and
$W^{2s,k'}(\mathbb{R}^N)$ we further have
\begin{equation}
|\int_{K}\phi^{k'-1}u(-\Delta)^{s}\phi dx|\leq
d_4\left(\int_{\Omega}\phi^{k'}u^k
dx\right)^{1/k}||\phi||_{W^{2s,k'}(\Omega)}.\label{ineq14}
\end{equation}
where $d_{4} > 0$. From ($\ref{ineq12}$) we have 
$$\mu(K)\leq \int_{\Omega}\phi^{k'}u^{k}w(x)dx+\int_{\Omega}\phi^{k'}u\;dx+\int_{\Omega}F[u]\phi^{k'}dx+d_4\left(\int_{\Omega}\phi^{k'}u^k
dx\right)^{1/k}||\phi||_{W^{2s,k'}(\Omega)}.$$ Let
$\text{Cap}_{s,k'}(K)=0$. Then by the definition of the capacity we
have a sequence of functions $\phi_n$ such that $0\leq \phi_n \leq
1$, $\phi_n\equiv 1$ and $||\phi_n||_{W^{2s,k'}(\Omega)}\rightarrow
0$. The Lebesgue measure of $K$ is zero since $\phi_n\equiv 1$ on
$K$ and $\phi_n\rightarrow 0$ a.e. Thus using these observations in
($\ref{ineq14}$) and passing the limit $n\rightarrow\infty$ we get
$\mu(K)=0$.\\
{\it Sufficient condition}~We begin by defining the truncation
$T_n(r)=\min\{n,|r|\}\text{sign}(r)$ and by assuming $\mu\in
W^{-2s,k}(\Omega)\bigcap\mathfrak{m}_{+}^b(\Omega)$. So, for each
$n\in\mathbb{N}$, we have
\begin{align}
(-\Delta)^su+g(T_n(u))&=\mu~\text{in}~\Omega\nonumber\\
u&= 0~\text{in}~\mathbb{R}^N\setminus\Omega,
\end{align}
where $g(T_n(u))=T_{n}(u)+F({T_{n}(u)}(x))+|T_n(u)|^{k-1}T_n(u)$. By
Theorem (1.1) in \cite{veron} there exists a non negative solution for each $n$.
Observe that $T_{n+1}(u)\geq T_{n}(u)$ and hence
$(T_{n+1}(u))^{k}\geq (T_{n+1}(u))^{k}$, $F(T_{n+1}(u)(x)) \geq
F((T_{n}(u)(x))$ to get $g(T_{n+1}(u)) \geq g(T_{n}(u))$. Further,
\begin{align}
(-\Delta)^su_n+g(T_{n+1}(u_n))&=\mu+g(T_{n+1}(u_n))-g(T_{n}(u_n))\nonumber\\
& \geq  \mu=(-\Delta)^su_{n+1}+g(T_{n+1}(u_{n+1})).\nonumber
\end{align}
So, by the comparison of solutions we have $u_n \geq u_{n+1}$. Set,
$\underset{n\rightarrow\infty}{\lim}u_n(x)=u(x)$. Therefore by the
Egoroff's theorem $u_n\rightarrow u$ a.e. in $\Omega$ and thus
$u_n\rightarrow u$ in $L^1(\Omega)$. Since $\mu\in
W^{-2s,k}(\Omega)$ this says that $\mathbb{G}[\mu]\in L^k(\Omega)$.
Since $0\leq u_n\leq \mathbb{G}[\mu]$ so $|T(u_n)|^k\leq
(\mathbb{G}[\mu])^k$ because $T_n(u_n)\rightarrow u$ a.e. in
$\Omega$. Hence by the dominated convergence theorem we have
$\underset{n\rightarrow\infty}{\lim}\int_{\Omega}T_n(u_n)^k=\int_{\Omega}u^k$.
Thus passing the limit $n\rightarrow\infty$ to
\begin{equation}
\int_{\Omega}(u_n(-\Delta)^s\xi+g(T_n(u))\xi)dx=\int_{\Omega}\xi
d\mu\nonumber
\end{equation}
for each $\xi\in X_s$. So we conclude that $u$ is a very weak unique solution to the problem (3.3) for
$\mu\in W^{-2s,k}(\Omega)\bigcap\mathfrak{m}_{+}^{b}(\Omega)$.\\
Now let $\mu$ be such that whenever for $K\subset\Omega$ compact
$\text{Cap}_{2s,k'}(K)=0\Rightarrow\mu(K)=0$. Then by the result due
to Feyel and de la Pradelle \cite{feyel} there exists an increasing
sequence of measures say $(\mu_n)\subset
W^{-2s,k}(\Omega)\bigcap\mathfrak{m}_{+}^b(\Omega)$ which converges
weakly to $\mu$. Therefore by the above argument $\exists v_n$ for
each $\mu_n$ such that
\begin{align}
(-\Delta)^sv_n+g(v_n)&=\mu_n\nonumber\\
v_n&=0~\text{in}~\mathbb{R}^N\setminus\Omega.\nonumber
\end{align}
in the very weak sense. Note that $(v_n)$ is an increasing sequence.
Choose $\eta_1$ as a particular test function which is a solution of
\begin{align}
(-\Delta)^s\eta_1&=1\nonumber\\
\eta_1&=0~\text{in}~\mathbb{R}^N\setminus\Omega,\nonumber
\end{align}
and which also has the property that
$c^{-1}\leq\frac{\eta_1}{\rho^s}\leq c$ for some $c>0$. Therefore
\begin{align*}
0\leq \int_{\Omega}(v_n+g(T_n(v_n))\eta_1)dx&=\int_{\Omega}\eta_1 d\mu_n \leq  \int_{\Omega}\eta_1 d\mu.\nonumber
\end{align*}
So we have $0\leq v_n\leq\mathbb{G}[\mu_n]$ and
$\mathbb{G}[\mu_n]\leq\mathbb{G}[\mu]$ from ($\ref{ineq3'}$). Hence,
$v_n\rightarrow v$ a.e. thus implying $v_n\rightarrow v$ in
$L^1(\Omega)$. It can also be seen that $v_n\rightarrow v$ in
$L^k(\Omega,\rho^sdx)$. We thus have $v$ is a solution to the
problem (3.3) for $\mu$ being a non negative and bounded measure.
\end{proof}
\section*{Acknowledgement}
The author Amita Soni thanks the Department of
Science and Technology (D. S. T), Govt. of India for financial
support. Both the authors also acknowledge the facilities received
from the Department of Mathematics, National Institute of Technology Rourkela.

\bibliographystyle{amsplain}

{\sc D. Choudhuri} and {\sc Amita Soni} \\
Department of Mathematics,\\
National Institute of Technology Rourkela, Rourkela - 769008,
India\\
e-mails: dc.iit12@gmail.com and soniamita72@gmail.com.
\end{document}